\documentclass[twoside,11pt]{amsart}
\usepackage{amsmath}
\usepackage{graphicx}
\usepackage{amsfonts}
\usepackage{amssymb}

\theoremstyle{plain}
\newtheorem{thm}{Theorem}[section]
\newtheorem{lem}[thm]{Lemma}
\newtheorem{prop}[thm]{Proposition}

\theoremstyle{definition}

\theoremstyle{remark}
\newtheorem*{rem}{Remark}

\begin{document}
\title[Boundary value problem for the Vlasov-Poisson system]{On global
existence of classical solutions for the Vlasov-Poisson system in convex
bounded domains}
\author{Hyung Ju Hwang}
\address{Department of Mathematics, Pohang 790-784, Republic of Korea}
\email{hjhwang@postech.ac.kr}
\author{Jaewoo Jung}
\address{Department of Mathematics, Pohang 790-784, Republic of Korea}
\email{zeujung@postech.ac.kr}
\author{Juan J. L. Vel\'{a}zquez}
\address{ICMAT (CSIC-UAM-UC3M-UCM),\ Universidad Complutense, Madrid 28035,
Spain.}
\email{velazque@mat.ucm.es}
\date{}

\begin{abstract}
We prove global existence of strong solutions for the Vlasov-Poisson system
in a convex bounded domain in the plasma physics case assuming homogeneous
Dirichlet boundary conditions for the electric potential and the specular
reflection boundary conditions for the distribution density.
\end{abstract}

\maketitle

\section{\protect\bigskip Introduction\protect\bigskip}

We consider the Vlasov-Poisson system in a smooth, convex and bounded domain
$\Omega $, with the reflection boundary condition given by
\begin{gather}
\partial _{t}f+v\cdot \nabla _{x}f+\nabla _{x}\phi \cdot \nabla
_{v}f=0,\quad \text{for $\left( t,x,v\right) \in \lbrack 0,\infty )\times
\Omega \times \mathbb{R}^{3}$,}  \label{S1E1} \\
\Delta \phi (t,x)=\rho (t,x)=\int_{\mathbb{R}^{3}}f(t,x,v)\,dv,\quad \text{%
for $(t,x)\in \lbrack 0,\infty )\times \Omega $,}  \label{S1E2} \\
f\left( 0,x,v\right) =f_{0}(x,v),~~~\text{for }\left( x,v\right) \in \Omega
\times \mathbb{R}^{3}  \label{S1E3} \\
f(t,x,v)=f(t,x,v^{\ast }),\quad \text{for $x\in \partial \Omega $,}
\label{S1E4}
\end{gather}%
where $v^{\ast }=v-2(v\cdot n_{x})n_{x}$, $n_{x}$ is the outward unit normal
vector to $\partial \Omega $ at $x\in \partial \Omega ,$ and we assume $%
\Omega $ has a $C^{5}$ boundary. Here $f\left( t,x,v\right) \geq 0$
represents the distribution density of electrons and $\phi $ is the
electrostatic potential, $f_{0}\left( x,v\right) $ is the prescribed initial
datum, and $\rho \left( t,x\right) $ is the macroscopic charge density. We
impose the Dirichlet boundary condition for the electric potential $\phi $,
\begin{equation}
\phi (t,x)=0\quad \text{if $x\in \partial \Omega $.}  \label{S1E5}
\end{equation}

In this paper, we aim to understand the role of boundaries in the dynamics
of kinetic models, in particular the Vlasov-Poisson system and to develop
tools for boundary-value problems in kinetic models.\bigskip

For the whole space without a boundary in one and two dimensions, a smooth
solution is known to exist globally in time \cite{SVI}, \cite{SUTO}. For the
three dimensional case, Batt \cite{JB}, Horst \cite{EH1}, and Bardos-Degond
\cite{CBPD} proved the global existence of classical solutions for
spherical, cylindrically symmetric, and general but small initial data
respectively. In the case of $\Omega =\mathbb{R}^{3}$ with arbitrary initial
data, the solutions of the system (\ref{S1E1})-(\ref{S1E4}) are globally
defined in time, as it was proved in \cite{KP} as well as in \cite{PLLBP}
using different methods.

However, in the presence of boundaries, the mathematical theory of
well-posedness for the solutions of the Vlasov-Poisson system becomes more
complicated compared to the case of the whole space. It was proved in \cite%
{YG1} that classical solutions for the Vlasov-Poisson system may not exist
in general without the nonnegativity assumption if $\Omega $ is the
half-space $\mathbb{R}_{+}^{3}$. On the other hand, it was also proved in
\cite{YG1} that even with the nonnegativity assumption the derivatives of
the solutions of (\ref{S1E1})-(\ref{S1E4}) cannot be uniformly bounded near
the boundary of $\Omega $ due to the fact that a Lipschitz estimate for the
characteristics in terms of the initial data is not possible.

One of the main difficulties in order to solve (\ref{S1E1})-(\ref{S1E5}),
even for short times, is to keep track of the evolution of the
characteristic curves associated to (\ref{S1E1}) which remain close during
their evolution to the so-called singular set, defined as follows
\begin{equation}
\Gamma \mathbb{=}\left\{ \left( x,v\right) \in \Omega \times \mathbb{R}%
^{3}:x\in \partial \Omega ,\ v\in T_{x}\partial \Omega \right\} ,
\label{singset}
\end{equation}%
where $T_{x}\partial \Omega \subset \mathbb{R}^{3}$ is the tangent plane to $%
\partial \Omega $ at the point $x.$

Boundary-value problems should be treated more carefully and difficulties
due to singularity formation at a boundary may be expected \cite{YG1}.
Global existence in a half-space of solutions of (\ref{S1E1})-(\ref{S1E3})
satisfying the specular reflection boundary condition (\ref{S1E4}) and
Neumann boundary condition was first proved by Guo (cf. \cite{YG2}) by
adapting a high velocity moment method in \cite{PLLBP}. The proof of the
global existence of solutions of (\ref{S1E1})-(\ref{S1E3}) satisfying (\ref%
{S1E4}) and the Dirichlet boundary condition (\ref{S1E5}) was recently
proved in \cite{HJHJJLV1}. For general convex bounded domains but with the
Neumann boundary condition for $\phi $, the global well-posedness was
recently shown in \cite{HJHJJLV2}.

\bigskip

However, a global existence theory of the Dirichlet boundary problem for the
electric potential $\phi $ has not been given yet and this paper is devoted
to proving the global existence of solutions to the Vlasov-Poisson system (%
\ref{S1E1})-(\ref{S1E3}) with the specular reflection boundary condition (%
\ref{S1E4}) for $f$ and the Dirichlet boundary condition (\ref{S1E5}) for $%
\phi $.

This paper combines the methods in the papers \cite{HJHJJLV1} and \cite%
{HJHJJLV2} to prove global existence of solutions for the Vlasov-Poisson
system in arbitrary smooth convex domains with Dirichlet boundary
conditions. The analysis in \cite{HJHJJLV2} allows to study problems with
Neumann boundary conditions. This is due to the fact that an essential
ingredient of the argument in \cite{HJHJJLV2} is the velocity Lemma first
proved in \cite{YG1} which shows that the characteristic curves associated
to (\ref{S1E1})-(\ref{S1E2}) cannot approach to the so-called singular set
if initially they are outside of it. The argument allows to generalize such
type of velocity Lemmas to Dirichlet boundary conditions, which was obtained
in \cite{HJHJJLV1}, but for the half-space case. However, simple adaptations
of such techniques in the half-space problem to the case of arbitrary convex
bounded domains do not work. The difficulty is that, contrary to the
half-space case \cite{HJHJJLV1}, we cannot have the representation formula
for $\phi $, due to the incapability of finding an explicit form of the
Green function for a general convex domain $\Omega $. As we will see, this
problem could be settled mainly by applying refined boundary estimates for
the Laplace operator, and constructing relevant supersolutions. \newline

This paper proves global existence of solutions for the Vlasov-Poisson
system for arbitrary smooth domains with Dirichlet boundary conditions. The
main new contents of this paper are some technical estimates that allow to
extend the arguments of \cite{HJHJJLV1} to arbitrary smooth convex domains.
These estimates require detailed control of the newtonian potential with
Dirichlet boundary conditions as well as some of its derivatives at points
close to the boundary of the domain. The combination of these methods with
ones in \cite{HJHJJLV2} allows to prove the stated global existence results.

The paper is organized as follows. Preliminary notations and main result of
the global existence will be presented in Section 2 and Section 3 is devoted
to the Velocity lemma and the corresponding linear problem. In Section 4, an
iterative scheme for the nonlinear problem is investigated and finally the
global bound on a key quantity for the global existence is obtained in
Section 5.

\bigskip

\section{Main Result}

First we fix a point $x\in \Omega $ and denote it as $\widetilde{x}$ to
indicate our target point. Notice that our goal is to see whether a
trajectory starting from the singular set $\{(x,v)\in \partial \Omega \times
\mathbb{R}^{3};v\cdot n_{x}=0\}$ propagates into the interior $\Omega \times
\mathbb{R}^{3}$. So, we may assume that $\widetilde{x}$ is near the boundary
$\partial \Omega $. Let $x_{0}$ be the boundary point closest to $\widetilde{%
x}$. By proper rotations and translations, we may set $x_{0}=(0,0,0)$, $%
\widetilde{x}=(\widetilde{x}_{1},0,0)$, and $\Omega \subset \mathbb{R}%
_{+}^{3}:=\{(x_{1},x_{2},x_{3})\in \mathbb{R}^{3};x_{1}>0\}$ so that the
tangent plane to $\partial \Omega $ at $x_{0}$ is just $\{x_{1}=0\}$.
\newline

Now, we use the local parametrization near $x_{0}$ to define
\begin{equation*}
x=x_{\parallel }(\mu _{1},\mu _{2})-x_{\perp }n(\mu _{1},\mu _{2}),
\end{equation*}%
so that $x_{\parallel }(\mu _{1},\mu _{2})$ is the point of $\partial \Omega
$ closest to $x$ and $n(\mu _{1},\mu _{2})$ is the outward normal to $%
\partial \Omega $ at $x_{\parallel }$. For this $x$, we represent $v$ by
\begin{equation*}
v=v_{\parallel }(\mu _{1},\mu _{2})-v_{\perp }n(\mu _{1},\mu _{2}),
\end{equation*}%
where $v_{\parallel }(\mu _{1},\mu _{2})=w_{1}u_{1}+w_{2}u_{2}\in
T_{x_{\parallel }(\mu _{1},\mu _{2})}\partial \Omega $ is the tangential
component of $v$ and $\{u_{1},u_{2}\}$ is the basis of $T_{x_{\parallel
}(\mu _{1},\mu _{2})}\partial \Omega $ given by $u_{i}:=\frac{\partial
x_{\parallel }(\mu _{1},\mu _{2})}{\partial \mu _{i}}$ for $i=1,2$. \newline

\bigskip

The system of coordinates $\left( \mu _{1},\mu _{2},x_{\perp
},w_{1},w_{2},v_{\perp }\right) $ provides a more convenient representation
for the set of points in the phase space $\Omega \times \mathbb{R}^{3}$ that
are close to the singular set $\Gamma $ defined in (\ref{singset}) as in
\cite{HJHJJLV2}. The original equation (\ref{S1E1}) takes in this new set of
coordinates the following form in Lemma below, which is in \cite{HJHJJLV2}
and we skip its proof.

\begin{lem}
\label{Lemma1}The equation (\ref{S1E1}) can be rewritten for $\left(
x,v\right) \in \left[ \partial \Omega +B_{\delta }\left( 0\right) \right]
\times \mathbb{R}^{3}$, and using the set of coordinates $\left( \mu
_{1},\mu _{2},x_{\perp },w_{1},w_{2},v_{\perp }\right) $ in the form
\begin{equation}
\frac{\partial f}{\partial t}+\sum_{i=1}^{2}\frac{w_{i}}{1+k_{i}x_{\perp }}%
\frac{\partial f}{\partial \mu _{i}}+v_{\perp }\frac{\partial f}{\partial
x_{\perp }}+\sum_{i=1}^{2}\sigma _{i}\frac{\partial f}{\partial w_{i}}+F%
\frac{\partial f}{\partial v_{\perp }}=0,  \label{S2geom}
\end{equation}%
where
\begin{equation}
\sigma _{i}\equiv E_{i}-\frac{v_{\perp }w_{i}k_{i}}{1+k_{i}x_{\perp }}%
-\sum_{j,\ell =1}^{2}\frac{\Gamma _{j,\ell }^{i}w_{j}w_{\ell }}{%
1+k_{j}x_{\perp }},\ \ F\equiv E_{\perp }+\sum_{j=1}^{2}\frac{w_{j}^{2}b_{j}%
}{1+k_{j}x_{\perp }},  \label{S2defs}
\end{equation}%
where $k_{j}$ are the principal curvatures, $b_{j}$\textbf{\ }are the
coefficients $e$ and $g$ from the second fundamental form according to the
notation in \cite{DJS} and $\Gamma _{j,\ell }^{i}$ are the Christoffel
symbols of the surface $\partial \Omega .$ The vector $E=\nabla _{x}\phi $
has been written in the form
\begin{equation}
E=E_{1}u_{1}+E_{2}u_{2}-E_{\perp }n\left( \mu _{1},\mu _{2}\right) ,
\label{Edecomp}
\end{equation}%
where $u_{1},u_{2}$ are defined above.
\end{lem}

\begin{rem}
Notice that since the domain $\Omega $ is convex, and due to the
nonnegativity assumption we have $F<0.$
\end{rem}

\bigskip

To prove global existence, we need to make some necessary assumptions.

1. \textbf{Compatibility conditions for the initial data.}

In order to obtain classical solutions of (\ref{S1E1})-(\ref{S1E5}) we need
to impose the following compatibility conditions on the initial data $%
f_{0}\left( x,v\right) $ at the reflection points of $\partial \Omega \times
\mathbb{R}^{3}$ (cf. \cite{YG1}, \cite{HJH}).
\begin{align}
f_{0}\left( x,v\right) & =f_{0}\left( x,v^{\ast }\right) ,  \label{S2E0b} \\
v^{\perp }\left[ \nabla _{x}^{\perp }f_{0}\left( x,v^{\ast }\right) +\nabla
_{x}^{\perp }f_{0}\left( x,v\right) \right] +2E^{\perp }\left( 0,x\right)
\nabla _{v}^{\perp }f_{0}\left( x,v\right) & =0,  \label{S2E0c}
\end{align}%
where $E^{\perp }\left( 0,x\right) $ is the decomposition of the field $%
E\left( 0,x\right) $ given by (\ref{Edecomp}) and $\nabla _{x}^{\perp
},\,\nabla _{v}^{\perp }$ are the normal components to $\partial \Omega $ of
the gradients $\nabla _{x},\;\nabla _{v}$ respectively.

\bigskip

2. \textbf{Flatness condition.}

We assume that $f_{0}$ is constant near the singular set (cf. \cite{YG2} as
well as \cite{HJH}). More precisely we will assume that $f_{0}\in C^{1,\mu }$
satisfies the following flatness condition near the singular set $\Gamma $

\begin{equation}
f_{0}\left( x,v\right) =\text{constant}\;,\;\text{dist}\left( \left(
x,v\right) ,\Gamma\right) \leq\delta_{0}  \label{S2flatness}
\end{equation}
for some $\delta>0$ small.

We need to introduce some functional spaces for technical reasons. We define
for $\mu \in \left( 0,1\right) ,$%
\begin{align*}
\left\Vert f\right\Vert _{C^{1,\mu }\left( \bar{\Omega}\times \mathbb{R}%
^{3}\right) }& =\sup_{\left( x,v\right) ,\left( x^{\prime },v^{\prime
}\right) \in \bar{\Omega}\times \mathbb{R}^{3}}\left( \frac{\left\vert
\nabla f\left( x,v\right) -\nabla f\left( x^{\prime },v^{\prime }\right)
\right\vert }{\left\vert x-x^{\prime }\right\vert ^{\mu }+\left\vert
v-v^{\prime }\right\vert ^{\mu }}\right) +\left\Vert f\right\Vert
_{L^{\infty }\left( \bar{\Omega}\times \mathbb{R}^{3}\right)
}\;\;,\;\;\nabla =\left( \nabla _{x},\nabla _{v}\right) ~, \\
C_{0}^{1,\mu }\left( \bar{\Omega}\times \mathbb{R}^{3}\right) & =\left\{
f\in C^{1,\mu }\left( \bar{\Omega}\times \mathbb{R}^{3}\right) :f\text{
compactly supported,\ }\left\Vert f\right\Vert _{C^{1,\mu }\left( \bar{\Omega%
}\times \mathbb{R}^{3}\right) }<\infty \right\} ~~,
\end{align*}%
\begin{align*}
\left\Vert f\right\Vert _{C_{\;t;\;x}^{1;1,\mu }\left( \left[ 0,T\right]
\times \bar{\Omega}\right) }& \equiv \sup_{x,x^{\prime }\in \bar{\Omega}%
,\;t,t^{\prime }\in \left[ 0,T\right] }\frac{\left\vert \nabla _{x}f\left(
t,x\right) -\nabla _{x}f\left( t^{\prime },x^{\prime }\right) \right\vert }{%
\left\vert x-x^{\prime }\right\vert ^{\mu }} \\
& +\left\Vert f\right\Vert _{C\left( \left[ 0,T\right] \times \bar{\Omega}%
\right) }+\left\Vert f_{t}\right\Vert _{C\left( \left[ 0,T\right] \times
\bar{\Omega}\right) }~\ ,
\end{align*}%
\begin{align*}
& \left\Vert f\right\Vert _{C_{t;\left( x,v\right) }^{1;1,\mu }\left( \left[
0,T\right] \times \Omega \times \mathbb{R}^{3}\right) } \\
& \equiv \sup_{x,x^{\prime }\in \bar{\Omega},\;t,t^{\prime }\in \left[ 0,T%
\right] }\frac{\left\vert \nabla _{x}f\left( t,x,v\right) -\nabla
_{x}f\left( t^{\prime },x^{\prime },v\right) \right\vert +\left\vert \nabla
_{v}f\left( t,x,v\right) -\nabla _{v}f\left( t^{\prime },x^{\prime
},v^{\prime }\right) \right\vert }{\left\vert x-x^{\prime }\right\vert ^{\mu
}+\left\vert v-v^{\prime }\right\vert ^{\mu }}+ \\
& +\left\Vert f\right\Vert _{C\left( \left[ 0,T\right] \times \bar{\Omega}%
\times \mathbb{R}^{3}\right) }+\left\Vert f_{t}\right\Vert _{C\left( \left[
0,T\right] \times \bar{\Omega}\times \mathbb{R}^{3}\right) }.
\end{align*}

We define the spaces $C\left( \left[ 0,T\right] \times \bar{\Omega}\right)
,\;C\left( \left[ 0,T\right] \times \bar{\Omega}\times \mathbb{R}^{3}\right)
$ as the spaces of continuous functions bounded in the uniform norm.

\bigskip

\bigskip

\textbf{The main result: Global existence Theorem.}

The main result of this paper is the following.

\begin{thm}
\label{globalexistence}Let $f_{0}\in C_{0}^{1,\mu }\left( \Omega \times
\mathbb{R}^{3}\right) $ for some $0<\mu <1$ with $f_{0}\geq 0$ and let $%
f_{0} $ satisfy (\ref{S2E0b})-(\ref{S2flatness}). Then there exists a unique
solution $f\in C_{t;\left( x,v\right) }^{1;1,\lambda }\left( \left( 0,\infty
\right) \times \Omega \times \mathbb{R}^{3}\right) $, $\phi \in
C_{t;x}^{1;3,\lambda }\left( \left[ 0,\infty \right) \times \Omega \right) ,$
for some $0<\lambda <\mu ,$ of the Vlasov-Poisson system (\ref{S1E1})-(\ref%
{S1E5}) with compact support in $x$ and $v.$
\end{thm}

\section{Velocity Lemma and Linear Problem}

Next, we introduce the evolution of characteristic curves associated to the
Vlasov-Poisson system with the specular reflection at the boundary. \newline

Let $E(t,x):=\nabla _{x}\phi (t,x)$ be given. We define $%
(X(s;t,x,v),V(s;t,x,v))\in \overline{\Omega }\times \mathbb{R}^{3}$ such
that for each $(x,v)\in \Omega \times \mathbb{R}^{3}$,
\begin{gather}
\frac{dX}{ds}(s;t,x,v)=V(s;t,x,v),  \label{S3E1} \\
\frac{dV}{ds}(s;t,x,v)=E(s,X(s;t,x,v))=\nabla _{x}\phi (s,X(s;t,x,v)),
\label{S3E2} \\
(X(t),V(t))=(x,v),  \label{S3E3}
\end{gather}%
as long as $X\in \Omega $. The reflection boundary condition says that if $%
X(s_{1};t,x,v)\in \partial \Omega $ for some $s_{1}\in \lbrack 0,T]$, then
\begin{equation}
V(s_{1}^{+};t,x,v)=\lim_{\substack{ s\rightarrow s_{1}  \\ s>s_{1}}}%
V(s;t,x,v)=(V(s_{1}^{-};t,x,v))^{\ast }=\biggl(\lim_{\substack{ s\rightarrow
s_{1}  \\ s<s_{1}}}V(s;t,x,v)\biggl)^{\ast }.  \label{S3E4}
\end{equation}%
Here, $V^{\ast }=V-2(V\cdot n_{X})n_{X}$ where $n_{X}$ is the outward unit
normal vector to $\partial \Omega $ at $X$. \newline

Before giving an explicit formulation, we consider some underlying
motivations. If we rephrase the Velocity Lemma, it is equivalent to saying
that a trajectory starting near the singular set $\{x_{\perp }=v_{\perp
}=0\} $ remains near it in the future. More precisely, by using the local
coordinates, we represent the normal component of the characteristic
equations from Lemma \ref{Lemma1} by
\begin{equation*}
\frac{dx_{\perp }}{dt}=v_{\perp },\qquad \frac{dv_{\perp }}{dt}=E_{\perp
}(t,x)+\sum_{i=1}^{2}\frac{w_{i}^{2}b_{i}}{1+k_{i}x_{\perp }},
\end{equation*}%
where $E_{\perp }(t,x)$ is the normal component of $E(t,x)$, $k_{i}$'s are
the principal curvatures, and $b_{i}$'s are the coefficients $e$ and $g$
from the second fundamental form, according to the notations in \cite{DJS}.
Notice that a trajectory cannot escape from the singular set, provided $\dot{%
v_{\perp }}<0$ near the boundary. Roughly, this is true because $E_{\perp
}<0 $ due to Hopf Lemma and $b_{i}\leq 0$ by the convexity of $\Omega $.
Finally, we define a Lyapunov function
\begin{equation*}
\alpha (t,x,v):=\frac{v_{\perp }^{2}}{2}-\phi (t,x)-\sum_{i=1}^{2}\frac{%
w_{i}^{2}b_{i}}{1+k_{i}x_{\perp }}x_{\perp }
\end{equation*}%
and confirm the stability by differentiating it along the trajectory. It may
be possible to choose another function which is equivalent to $x_{\perp
}+v_{\perp }^{2}$, but the functional $\alpha $ makes computations simpler
because of the cancellations. \newline

Now, we begin to show the following Lemmas which will be the crucial
estimations required to derive our main result. Recall that $\Omega $ and $%
\widetilde{x}$ are given in Section 2.

\begin{lem}
\label{timederivative} Let $T>0$. Suppose that $\phi (t,x)$ solves the
following boundary value problem
\begin{gather*}
\Delta \phi (t,x)=\rho (t,x)\quad \text{for $(t,x)\in \lbrack 0,T]\times
\Omega $,} \\
\phi (t,x)=0\quad \text{if $x\in \Omega $,}
\end{gather*}%
where $\rho \in C^{1}([0,T]\times \Omega )$ is given by
\begin{equation}
\partial _{t}\rho +\nabla \cdot j=0\quad \text{in $[0,T]\times \Omega $,}
\label{localmassconservation}
\end{equation}%
for some $j\in (C^{1}([0,T]\times \Omega ))^{3}$. Then we have
\begin{equation*}
\biggl\lvert\frac{\partial \phi }{\partial t}(t,\widetilde{x})\biggl\lvert%
\leq C\widetilde{x}_{1}(1+|\log \widetilde{x}_{1}|),
\end{equation*}%
where $C>0$ depends only on $L:=\mathrm{diam}\Omega $ and $|\!|j|\!|_{\infty
}$.
\end{lem}

\begin{proof}
Let $R=|\widetilde{x}|\ll 1$. We change the variables $x$ and $y$ by $X=%
\frac{x}{R}$ and $Y=\frac{y}{R}$. Also, $\widetilde{X}=\frac{\widetilde{x}}{R%
}=(1,0,0)$. Let $G$ be the Green function for the given domain $\Omega $.
Then, by the representation formula and \eqref{localmassconservation}, it is
sufficient to show that
\begin{equation}
\int_{\frac{\Omega }{R}}|\nabla _{Y}G(\widetilde{X},Y)|\,dY\leq \frac{C}{R}%
(1+|\log R|).  \label{claim1}
\end{equation}

We split the region of integration into several parts.

\textbf{Case 1.} If $|Y|\leq 4$, then we decompose $G(X,Y)=\overline{G}%
(X,Y)+W(X,Y)$ where
\begin{equation*}
\overline{G}(X,Y)=-\frac{1}{4\pi R}\biggl(\frac{1}{|X-Y|}-\frac{1}{|X^{\ast
}-Y|}\biggl)
\end{equation*}%
is the Green function for the Half-space restricted to $\frac{\Omega }{R}%
\times \frac{\Omega }{R}$. Here, $X^{\ast }$ represents the reflection of $X$
with respect to the plane $\{X_{1}=0\}$. Clearly, we have
\begin{equation*}
\int_{|Y|\leq 4}|\nabla _{Y}\overline{G}(X,Y)|\,dY\leq \frac{C}{R}.
\end{equation*}%
On the other hand, let $\Psi =\nabla _{Y}W$. Then, $\Psi $ satisfies
\begin{gather*}
\Delta _{X}\Psi (X,Y)=0\quad \text{for $X,Y\in \textstyle\frac{\Omega }{R}$,}
\\
\Psi (X,Y)=-\frac{1}{4\pi R}\nabla _{Y}\biggl(\frac{1}{|X-Y|}-\frac{1}{%
|X^{\ast }-Y|}\biggl)\quad \text{for $X\in \textstyle\frac{\partial \Omega }{%
R}$}.
\end{gather*}%
Firstly, if $\mathrm{dist}(Y,\frac{\partial \Omega }{R})\geq 1$, then for
all $X\in \frac{\partial \Omega }{R}$, we have
\begin{equation*}
|\Psi (X,Y)|\leq \frac{C|X-X^{\ast }|}{R|X-Y|^{3}}\leq \frac{C|X|^{2}}{%
|X|^{3}+1}\leq C,
\end{equation*}%
by Taylor Theorem and the quadratic approximations of $\frac{\partial \Omega
}{R}$. Notice that the constant $C>0$ can be chosen uniformly with respect
to $Y$. Thus,
\begin{equation*}
|\Psi (X,Y)|\leq C\quad \text{for all $X\in \textstyle\frac{\Omega }{R}$},
\end{equation*}%
by the maximum principle.

Secondly, if $\mathrm{dist}(Y,\frac{\partial \Omega }{R})\leq 1$, let $%
Y_{0}\in \frac{\partial \Omega }{R}$ be the boundary point closest to $Y$,
i.e., $\mathrm{dist}(Y,Y_{0})=\mathrm{dist}(Y,\frac{\partial \Omega }{R})$.
Then for $X\in \frac{\partial \Omega }{R}$, we get
\begin{equation*}
|\Psi (X,Y)|\leq \frac{C}{R|X-Y|^{2}}\quad \text{if $|X-Y_{0}|\geq 1$ or $%
\leq CR$,}
\end{equation*}%
by the triangle inequality and the convexity of $\Omega $. And, we have
\begin{equation*}
|\Psi (X,Y)|\leq \frac{C}{|X-Y_{0}|^{3}}\quad \text{if $CR\leq |X-Y_{0}|\leq
1$,}
\end{equation*}%
by Taylor Theorem. Putting these together, we let
\begin{equation*}
\widetilde{\Psi }(X,Y)=%
\begin{cases}
\displaystyle\frac{C}{R} & \text{if $|X-Y_{0}|\geq 1$,} \\
\displaystyle\frac{C}{R|X-Y|^{2}} & \text{if $|X-Y_{0}|\leq CR$,} \\
\displaystyle\frac{C}{|X-Y_{0}|^{3}} & \text{if $CR\leq |X-Y_{0}|\leq 1$.}%
\end{cases}%
\end{equation*}%
Now, for $|Y|\leq 4$ and $\mathrm{dist}(Y,\frac{\partial \Omega }{R})\leq 1$%
, by the maximum principle, we construct a supersolution via the Poisson
integral formula,
\begin{equation*}
\begin{split}
|\Psi (\widetilde{X},Y)|& \leq \int_{\frac{\partial \Omega }{R}}\frac{1}{%
(1+|\xi |^{2})^{\frac{3}{2}}}\widetilde{\Psi }(\xi ,Y)\,d^{2}\xi \\
& =\int_{|\xi -Y_{0}|\leq CR}+\int_{CR\leq |\xi -Y_{0}|\leq 1}+\int_{|\xi
-Y_{0}|\geq 1}.
\end{split}%
\end{equation*}%
Let $\eta :=Y-Y_{0}$. Then $|X-Y|=|(X-Y_{0})-\eta |$. Using the inequality
\begin{equation*}
\begin{split}
\int_{|\xi -Y_{0}|\leq CR}\frac{1}{(1+|\xi |^{2})^{\frac{3}{2}}}\frac{C}{%
R|(\xi -Y_{0})-\eta |^{2}}\,d^{2}\xi & \leq \frac{C}{R}\int_{|\xi |\leq CR}%
\frac{1}{|\xi |^{2}+|\eta |^{2}}\,d^{2}\xi \\
& \leq \frac{C}{R}\biggl\lvert\log \frac{CR}{\mathrm{dist}(Y,\frac{\partial
\Omega }{R})}\biggl\lvert,
\end{split}%
\end{equation*}%
we obtain
\begin{equation*}
|\Psi (\widetilde{X},Y)|\leq \frac{C}{R}\biggl\lvert\log \frac{CR}{\mathrm{%
dist}(Y,\frac{\partial \Omega }{R})}\biggl\lvert+\frac{C}{R}+C,
\end{equation*}%
and
\begin{equation*}
\int_{\substack{ \mathrm{dist}(Y,\frac{\partial \Omega }{R})\leq 1  \\ %
|Y|\leq 4}}|\Psi (\widetilde{X},Y)|\,dY\leq \frac{C}{R}(1+|\log R|).
\end{equation*}

\textbf{Case 2.} For $|Y|\geq 4$, we fix a point $Y=Y_{0}$. Rescale the
variables by $\eta =\frac{Y}{|Y_{0}|}$ and $\xi =\frac{X}{|Y_{0}|}$ so that $%
\xi $, $\eta \in \frac{\overline{\Omega }}{|Y_{0}|R}$. Define $g(\xi ,\eta
):=G(|Y_{0}|\xi ,|Y_{0}|\eta )=G(X,Y)$. Since $\Delta _{X}G(X,Y)=\delta
(X-Y) $, we have
\begin{equation*}
\Delta _{\xi }g(\xi ,\eta )=\frac{1}{|Y_{0}|}\delta (\xi -\eta ),
\end{equation*}%
by a change of variables. Let $\varphi (\xi ,\eta ):=|Y_{0}|\nabla _{\eta
}g(\xi ,\eta )$. Then we have
\begin{gather*}
\Delta _{\xi }\varphi (\xi ,\eta )=\nabla _{\eta }\delta (\xi -\eta )\quad
\text{for all $\xi $, $\eta \in \textstyle\frac{\Omega }{|Y_{0}|R}$}, \\
\varphi (\xi ,\eta )=0\quad \text{if $\xi \in \textstyle\frac{\partial
\Omega }{|Y_{0}|R}$}.
\end{gather*}%
Now, we again divide into two cases.

\begin{itemize}
\item[$\bullet $] If $\mathrm{dist}(\eta ,\frac{\partial \Omega }{|Y_{0}|R}%
)\geq \frac{1}{10}$, we define $\psi (\xi ,\eta )$ satisfying
\begin{equation*}
\varphi (\xi ,\eta )=-\frac{1}{4\pi }\nabla _{\eta }\frac{1}{|\xi -\eta |}%
+\psi (\xi ,\eta ).
\end{equation*}%
Since $\Delta _{\xi }\varphi =\nabla _{\eta }\delta (\xi -\eta )$, we have
\begin{gather*}
\Delta _{\xi }\psi (\xi ,\eta )=0\quad \text{for all $\xi $, $\eta \in %
\textstyle\frac{\Omega }{|Y_{0}|R}$}, \\
\psi (\xi ,\eta )=\frac{1}{4\pi }\nabla _{\eta }\frac{1}{|\xi -\eta |}\quad
\text{if $\xi \in \textstyle\frac{\partial \Omega }{|Y_{0}|R}$}.
\end{gather*}%
By assumption, $|\psi (\xi ,\eta )|\leq C$ for all $\xi \in \frac{\partial
\Omega }{|Y_{0}|R}$. By the maximum principle, we have
\begin{equation*}
|\psi (\xi ,\eta )|\leq C\quad \text{for all $\xi \in \textstyle\frac{\Omega
}{|Y_{0}|R}$}.
\end{equation*}%
If we apply the boundary regularity theory to the restricted region $\{\xi
\in \frac{\Omega }{|Y_{0}|R};|\xi |\leq \frac{1}{2}\}$ with $|\eta |=1$
fixed, then we have
\begin{equation*}
|\nabla _{\xi }\psi (\xi ,\eta )|\leq C\quad \text{for all $\xi \in %
\textstyle\frac{\Omega }{|Y_{0}|R}$}.
\end{equation*}%
Hence, we obtain
\begin{equation*}
|\varphi (\widetilde{\xi },\eta )|\leq C\mathrm{dist}(\widetilde{\xi },\frac{%
\partial \Omega }{|Y_{0}|R})=\frac{C}{|Y_{0}|},
\end{equation*}%
and
\begin{equation*}
|\nabla _{Y}G(\widetilde{X},Y_{0})|\leq \frac{C}{|Y_{0}|^{3}}.
\end{equation*}

\item[$\bullet $] If $\mathrm{dist}(\eta ,\frac{\partial \Omega }{|Y_{0}|R}%
)\leq \frac{1}{10}$, we fix $\eta $, and say $\eta _{0}$. Let
\begin{equation*}
\xi ^{\star }=\xi +2\mathrm{dist}(\xi ,\{\eta \in \mathbb{R}^{3};(\eta -%
\overline{\eta }_{0})\cdot \nu (\overline{\eta }_{0})=0\})\nu (\overline{%
\eta }_{0}),
\end{equation*}%
where $\overline{\eta }_{0}$ is the boundary point closest to $\eta _{0}$,
and $\nu (\overline{\eta }_{0})$ is the outward normal vector at $\overline{%
\eta }_{0}$. This means that $\xi ^{\star }$ is the reflection of $\xi $
with respect to the tangent plane at $\overline{\eta }_{0}$. We define $%
w(\xi ,\eta )$ such that it satisfies
\begin{equation*}
\nabla _{\xi }g(\xi ,\eta )=-\frac{1}{4\pi |Y_{0}|}\nabla _{\eta }\biggl(%
\frac{1}{|\xi -\eta |}-\frac{1}{|\xi ^{\star }-\eta |}\biggl)+w(\xi ,\eta ).
\end{equation*}%
Then we have
\begin{gather*}
\Delta _{\xi }w(\xi ,\eta )=0\quad \text{for all $\xi $, $\eta \in \textstyle%
\frac{\Omega }{|Y_{0}|R}$}, \\
w(\xi ,\eta )=\frac{1}{4\pi |Y_{0}|}\biggl(\frac{\xi -\eta }{|\xi -\eta |^{3}%
}-\frac{\xi ^{\star }-\eta }{|\xi ^{\star }-\eta |^{3}}\biggl)\quad \text{%
for $\xi \in \textstyle\frac{\partial \Omega }{|Y_{0}|R}$}.
\end{gather*}%
Now, for $\xi \in \frac{\partial \Omega }{|Y_{0}|R}$, we have
\begin{equation*}
|w(\xi ,\eta _{0})|\leq \frac{C}{|Y_{0}|}\quad \text{if $|\xi -\overline{%
\eta }_{0}|\geq \frac{1}{8}$},
\end{equation*}%
by the triangle inequality, and we have
\begin{equation*}
\begin{split}
|w(\xi ,\eta _{0})|& \leq \frac{CR|\xi -\overline{\eta }_{0}|^{2}}{|\xi
-\eta _{0}|^{3}} \\
& \leq \frac{CR|\xi -\overline{\eta }_{0}|}{|\xi -\overline{\eta }%
_{0}|^{3}+|\eta _{0}-\overline{\eta }_{0}|^{3}}\quad \text{if $|\xi -%
\overline{\eta }_{0}|\leq \frac{1}{8}$},
\end{split}%
\end{equation*}%
by Taylor Theorem. Let $D:=\mathrm{dist}(\eta _{0},\frac{\partial \Omega }{%
|Y_{0}|R})$. We again use the Poisson kernel estimation to get, for $\xi \in
B_{\frac{1}{|Y_{0}|}}(0)\cap \frac{\Omega }{|Y_{0}|R}$,
\begin{equation*}
\begin{split}
|w(\xi ,\eta _{0})|& \leq \int_{|z^{\prime }|\geq \frac{1}{8}}\frac{\xi _{1}-%
\overline{\eta }_{0,1}}{\{(\xi _{1}-\overline{\eta }_{0,1})^{2}+|(\xi
^{\prime }-\overline{\eta }_{0}^{\prime })-z^{\prime }|^{2}\}^{\frac{3}{2}}}%
\frac{C}{|Y_{0}|}\,d^{2}z^{\prime } \\
& \indent+\int_{|z^{\prime }|\leq \frac{1}{8}}\frac{\xi _{1}-\overline{\eta }%
_{0,1}}{\{(\xi _{1}-\overline{\eta }_{0,1})^{2}+|(\xi ^{\prime }-\overline{%
\eta }_{0}^{\prime })-z^{\prime }|^{2}\}^{\frac{3}{2}}}\frac{CR|z^{\prime }|%
}{(|z^{\prime }|^{3}+D^{3})}\,d^{2}z^{\prime } \\
& \leq \frac{C}{|Y_{0}|}+\int_{|Z^{\prime }|\leq \frac{1}{8D}}\frac{%
CR|Z^{\prime }|}{|Z^{\prime }|^{3}+1}\,d^{2}Z^{\prime } \\
& \leq C\biggl(\frac{1}{|Y_{0}|}+R|\log \mathrm{dist~}(\eta _{0},\textstyle%
\frac{\partial \Omega }{|Y_{0}|R})|\biggl).
\end{split}%
\end{equation*}%
Here, we use the prime notation to indicate the second and third components
in the Cartesian coordinate system. Similarly to the former case, we apply
the boundary regularity theory and obtain
\begin{equation*}
|\nabla _{\xi }w(\xi ,\eta _{0})|\leq C\biggl(\frac{1}{|Y_{0}|}+R|\log
\mathrm{dist~}(\eta _{0},\textstyle\frac{\partial \Omega }{|Y_{0}|R})|\biggl)%
,
\end{equation*}%
for all $\xi \in B_{\frac{1}{|Y_{0}|}}(0)\cap \frac{\Omega }{|Y_{0}|R}$.
Hence, by the change of variables, we get
\begin{equation*}
|\nabla _{Y}G(\widetilde{X},Y_{0})|\leq C\biggl(\frac{1}{|Y_{0}|^{2}}+\frac{R%
}{|Y_{0}|^{2}}|\log \mathrm{dist~}(\frac{Y_{0}}{|Y_{0}|},\textstyle\frac{%
\partial \Omega }{|Y_{0}|R})|\biggl).
\end{equation*}
\end{itemize}

In conclusion, if $|Y|\geq 4$, we have
\begin{multline*}
\int_{4\leq |Y|\leq \frac{L}{R}}|\nabla _{Y}G(\widetilde{X},Y)|\,d^{3}Y \\
\leq C\int_{4\leq |Y|\leq \frac{L}{R}}\biggl(\frac{1}{|Y|^{2}}+\frac{R}{%
|Y|^{2}}|\log \mathrm{dist~}(\textstyle\frac{Y}{|Y|},\frac{\partial \Omega }{%
|Y|R})|\biggl)\,d^{3}Y.
\end{multline*}%
For the second term, we get
\begin{multline*}
\int_{4\leq |Y|\leq \frac{L}{R}}\frac{R}{|Y|^{2}}|\log \mathrm{dist~}(\frac{Y%
}{|Y|},\frac{\partial \Omega }{|Y_{0}|R})|\,d^{3}Y \\
\begin{split}
& \leq C\sum_{n=0}^{\frac{|\log R|}{\log 2}}\int_{4\cdot 2^{n}\leq |Y|\leq
4\cdot 2^{n+1}}\frac{R}{|Y|^{2}}|\log \mathrm{dist~}(\textstyle\frac{Y}{|Y|},%
\frac{\partial \Omega }{|Y_{0}|R})|\,d^{3}Y \\
& \leq CR\sum_{n=0}^{\frac{|\log R|}{\log 2}}2^{n}\int_{4\leq |Z|\leq
8}|\log \mathrm{dist~}(\textstyle\frac{Z}{|Z|},\frac{\partial \Omega }{%
2^{n}R|Z|})|\,d^{3}Z. \\
& \leq C
\end{split}
\\
\end{multline*}%
Notice that the last integration is bounded. Thus, we have established %
\eqref{claim1} for $|Y|\geq 4$.
\end{proof}

\begin{lem}
\label{tangentialderivative} With the same assumptions of Lemma \ref%
{timederivative}, we have
\begin{equation*}
\biggl\lvert\frac{\partial \phi }{\partial x_{2}}(t,\widetilde{x})%
\biggl\lvert+\biggl\lvert\frac{\partial \phi }{\partial x_{3}}(t,\widetilde{x%
})\biggl\lvert\leq C\widetilde{x}_{1}(1+|\log \widetilde{x}_{1}|),
\end{equation*}%
where $C>0$ depends only on $L$ and $|\!|\rho |\!|_{\infty }$.
\end{lem}

\begin{proof}
As in the proof of Lemma \ref{timederivative}, we take the scaled variables $%
X=\frac{x}{R}$ and $Y=\frac{y}{R}$ where $R=|\widetilde{x}|$. Say $%
\widetilde{X}=\frac{\widetilde{x}}{R}=(1,0,0)$. To compute $|\frac{\partial G%
}{\partial x_{2}}|$, we divide it into two cases.

\textbf{Case 1.} If $|Y|\geq 2$, we decompose the Green function $G(X,Y)=%
\overline{G}(X,Y)+W(X,Y)$ where
\begin{equation*}
\overline{G}(X,Y)=-\frac{1}{4\pi R}\biggl(\frac{1}{|X-Y|}-\frac{1}{%
|X-Y^{\ast }|}\biggl)
\end{equation*}%
and $Y^{\ast }$ represents the reflection of $Y$ with respect to the plane $%
\{X_{1}=0\}$. Notice that if $|X|\leq \frac{3}{4}|Y|$, then we have
\begin{equation*}
|\overline{G}(X,Y)|\leq \frac{C}{R|Y|}.
\end{equation*}%
Moreover, since $0\leq W(X,Y)\leq -\overline{G}(X,Y)$, we get
\begin{equation*}
|G(X,Y)|\leq \frac{C}{R|Y|}.
\end{equation*}%
Now, we take the variables $\xi =\frac{X}{|Y|}$ and $\eta =\frac{Y}{|Y|}$
and consider the restricted region $\Omega _{0}=\{\xi \in \frac{\Omega }{R|Y|%
};|\xi |\leq \frac{3}{4}\}$ with $Y$ fixed. Since
\begin{equation*}
\Delta _{\xi }G(\xi ,\eta )=0\text{ in $\Omega _{0}$},\quad |G(\xi ,\eta
)|\leq \frac{C}{R|Y|}\text{ on $\partial \Omega _{0}$},
\end{equation*}%
applying regularity theory leads to
\begin{equation*}
\biggl\lvert\frac{\partial ^{\alpha }G}{\partial \xi ^{\alpha }}(\xi ,\eta )%
\biggl\lvert\leq \frac{C}{R|Y|}\quad \text{for any multi-index $\alpha $}.
\end{equation*}%
Let $\widetilde{\xi }=\frac{\widetilde{x}}{|Y|}$. Since $|\frac{\partial G}{%
\partial \xi _{2}}(0,\eta )|=0$, we have
\begin{equation*}
\biggl\lvert\frac{\partial G}{\partial \xi _{2}}(\widetilde{\xi },\eta )%
\biggl\lvert\leq \frac{C}{R|Y|^{2}},
\end{equation*}%
and
\begin{equation*}
\biggl\lvert\frac{\partial G}{\partial X_{2}}(\widetilde{X},Y)\biggl\lvert%
\leq \frac{C}{R|Y|^{3}}.
\end{equation*}

\textbf{Case 2.} If $|Y|\leq 2$, we denote as $Y_{0}$ the boundary point
closest to $Y$. Decompose $G(X,Y)=\overline{G}(X,Y)+W(X,Y)$ where
\begin{equation*}
\overline{G}(X,Y)=-\frac{1}{4\pi R}\biggl(\frac{1}{|X-Y|}-\frac{1}{|X-%
\overline{Y}|}\biggl).
\end{equation*}%
Here, $\overline{Y}$ is the reflection point of $Y$ with respect to the
tangent plane at $Y_{0}$. For $X\in \frac{\partial \Omega }{R}$, we express $%
W(X,Y)$ in terms of $e_{\perp }:=\frac{Y-Y_{0}}{|Y-Y_{0}|}$, $\eta :=\frac{%
X-Y_{0}}{|Y-Y_{0}|},$ and $\eta ^{\ast }:=\frac{\overline{X}-Y_{0}}{|Y-Y_{0}|%
}$ as
\begin{equation*}
W(X,Y)=\frac{1}{4\pi R|Y-Y_{0}|}\biggl(\frac{1}{|e_{\perp }-\eta |}-\frac{1}{%
|e_{\perp }-\eta ^{\ast }|}\biggl).
\end{equation*}%
Then we have
\begin{equation*}
\biggl\lvert\frac{1}{|e_{\perp }-\eta |}-\frac{1}{|e_{\perp }-\eta ^{\ast }|}%
\biggl\lvert\;%
\begin{cases}
\displaystyle\;\leq CR|Y-Y_{0}||\eta |^{2} & \text{if $|\eta |\leq 1$}, \\
&  \\
\displaystyle\;\leq \frac{CR|Y-Y_{0}|}{|\eta |} & \text{if $|\eta |\geq 1$}.%
\end{cases}%
\end{equation*}%
Combining these, we obtain $|W(X,Y)|\leq C$. On the other hands, for $X\in
\frac{\Omega }{R}$ on the line segment $|X|=4$, we have $|W(X,Y)|\leq \frac{C%
}{R}$. Indeed, this can be done by taking $-\overline{G}$ as a
supersolution, and then applying the maximum principle. Now, if we consider
the region $\frac{\Omega }{R}\cap \{|X|\leq 4\}$ with $Y$ fixed, then using
the regularity theory leads to
\begin{equation*}
|\nabla _{X}W(X,Y)|\leq \frac{C}{R},
\end{equation*}%
and by adding $|\nabla _{X}\overline{G}|$ term, we have
\begin{equation*}
\biggl\lvert\frac{\partial G}{\partial X_{2}}(\widetilde{X},Y)\biggl\lvert%
\leq \frac{C}{R|\widetilde{X}-Y|^{2}}.
\end{equation*}%
From all these calculations, we conclude
\begin{equation*}
\begin{split}
\int_{\frac{\Omega }{R}}\biggl\lvert\frac{\partial G}{\partial X_{2}}(%
\widetilde{X},Y)\biggl\lvert\,dY& =\int_{|Y|\geq 2}+\int_{|Y|\leq 2}%
\biggl\lvert\frac{\partial G}{\partial X_{2}}(\widetilde{X},Y)\biggl\lvert%
\,dY \\
& \leq \int_{|Y|\geq 2}\frac{C}{R|Y|^{3}}\,dY+\int_{|Y|\leq 2}\frac{C}{R|%
\widetilde{X}-Y|^{2}}\,dY \\
& \leq \frac{C}{R}(1+|\log R|).
\end{split}%
\end{equation*}
\end{proof}

We now give the main result, which plays the role of Velocity Lemma in our
setting.

\begin{lem}[Velocity Lemma]
\label{velocity}Let $(X(s;t,x,v),V(s;t,x,v))$ be the characteristic curves
associated to the Vlasov-Poisson system defined previously. Suppose $\phi
(t,x)$ satisfies the assumptions of Lemma \ref{timederivative}. Then there
exist constants $C_{1}$ and $C_{2}>0$ depending only on $\Omega $, $|\!|\rho
|\!|_{L^{\infty }}$, and $|\!|j|\!|_{L^{\infty }}$, such that if $X_{\perp }$
is small enough, we have
\begin{equation*}
C_{1}(X_{\perp }+V_{\perp }^{2})(t)\leq (X_{\perp }+V_{\perp }^{2})(s)\leq
C_{2}(X_{\perp }+V_{\perp }^{2})(t),
\end{equation*}%
for $s$, $t\in \lbrack 0,T]$.
\end{lem}

\begin{proof}
Due to Hopf Lemma, we can choose the constant $\epsilon _{0}>0$ such that
\begin{equation}
\phi (t,x)\leq -\epsilon _{0}x_{\perp }  \label{gradientestimate}
\end{equation}%
for $x_{\perp }$ small. We define
\begin{equation*}
\alpha (t,x,v)=\frac{v_{\perp }^{2}}{2}-\phi (t,x)-\biggl(\sum_{i=1}^{2}%
\frac{w_{i}^{2}b_{i}}{1+k_{i}x_{\perp }}\biggl)x_{\perp },
\end{equation*}%
where $b_{i}$'s are the coefficients of the second fundamental form, and $%
k_{i}$'s are the principal curvatures of the surface $\partial \Omega $.
Notice that $b_{i}\leq 0$ by the convexity of $\Omega $. So, $\alpha (t,x,v)$
is equivalent to $x_{\perp }+v_{\perp }^{2}$. That is, it is sufficient to
show
\begin{equation*}
C_{1}\alpha (t,X(t),V(t))\leq \alpha (s,X(s),V(s))\leq C_{2}\alpha
(t,X(t),V(t))
\end{equation*}%
for $s,t\in \lbrack 0,T]$. By differentiating $\alpha $ with respect to $t$
along the characteristics and by representing the field $E$ as $E=\nabla
_{x}\phi =E_{1}u_{1}+E_{2}u_{2}-E_{\perp }n_{x}(\mu _{1},\mu _{2})$, we have
\begin{equation*}
\begin{split}
\frac{d\alpha }{dt}(t,x(t),v(t))& =-\frac{\partial \phi }{\partial t}%
(t,x)-\sum_{i}w_{i}E_{i}\frac{b_{i}}{k_{i}} \\
& \indent-\sum_{i}\biggl(2w_{i}b_{i}E_{i}+w_{i}^{2}\frac{db_{i}}{dt}%
-2\sum_{j,l}\frac{\Gamma _{jl}^{i}w_{i}w_{j}w_{l}b_{i}}{1+k_{j}x_{\perp }}%
\biggl)\frac{x_{\perp }}{1+k_{i}x_{\perp }} \\
& \indent+\sum_{i}\biggl(v_{\perp }k_{i}-x_{\perp }\frac{dk_{i}}{dt}\biggl)%
\frac{x_{\perp }w_{i}^{2}b_{i}}{(1+k_{i}x_{\perp })^{2}},
\end{split}%
\end{equation*}%
where $\Gamma _{jl}^{i}$;s are the Christoffel symbols. Using Lemma \ref%
{timederivative}, \ref{tangentialderivative}, and the equation %
\eqref{gradientestimate}, we obtain
\begin{equation*}
\biggl\lvert\frac{d\alpha }{dt}(t,X(t),V(t))\biggl\lvert\leq C\alpha
(1+|\log \alpha |).
\end{equation*}%
Therefore, the Lemma follows by Gronwall inequality.
\end{proof}

\bigskip

We give the theorem of well-posedness for the linear problem (\ref{S1E1}), (%
\ref{S1E3}), (\ref{S1E4}) in the following theorem.

\begin{thm}
\label{Thlinear} Assume that $E\in C_{\;t;\;x}^{0;1,\mu }\left( \left[ 0,T%
\right] \times \bar{\Omega}\right) $ for some $\mu \in \left( 0,1\right) .$
Suppose that $f_{0}\in C_{0}^{1,\mu }\left( \bar{\Omega}\times \mathbb{R}%
^{3}\right) $ for some $\mu >0$ \ and $f_{0}\geq 0$ . Then there exists a
unique solution, $f\in $ $C_{t;\left( x,v\right) }^{1;1,\lambda }\left( %
\left[ 0,T\right] \times \Omega \times \mathbb{R}^{3}\right) ,$ to the
linear Vlasov-Poisson system (\ref{S1E1}), (\ref{S1E3}), (\ref{S1E4}), for
some $0<\lambda <\mu .$ Moreover the function $f$ satisfies
\begin{align}
f& \geq 0\;  \label{S3Xesp1} \\
\int f\left( t,x,v\right) dxdv& =\int f_{0}\left( x,v\right)
dxdv\;\;,\;\;t\in \left[ 0,T\right] .  \label{S3Xesp2n}
\end{align}
\end{thm}

\begin{proof}
The key point in the proof is that the characteristics (\ref{S3E1})-(\ref%
{S3E4}) intersect the boundary $\partial \Omega \times \mathbb{R}^{3}$ at
most a finite number of times and they never intersect with the singular
set, which is due to Velocity lemma and Lemma \ref{velocity}. The essential
procedure is similar to the proof of Theorem 2 in \cite{HJHJJLV2}, we skip
the details of the proof.
\end{proof}

\bigskip

\section{Iterative approach for the nonlinear problem}

In this section, we will show the global existence of classical solutions to
the fully nonlinear Vlasov-Poisson system (\ref{S1E1})-(\ref{S1E5}). Since
the procedures are similar to ones in \cite{HJHJJLV2}, we will not try to
give every detail of the proofs and instead we refer to \cite{HJHJJLV2}
whenever we need.

\subsection{Iterative procedure}

We will obtain a solution of the nonlinear system (\ref{S1E1})-(\ref{S1E5})
as the limit of a sequence of functions $f^{n}$ that are defined by an
iterative procedure. More precisely, we define
\begin{equation}
f^{0}\left( t,x,v\right) =f_{0}\left( x,v\right) \;\;,\;\;t\geq 0,\;x\in
\Omega ,\;v\in \mathbb{R}^{3}  \label{itf0}
\end{equation}%
\begin{align}
f_{t}^{n}+v\cdot \nabla _{x}f^{n}+\nabla _{x}\phi ^{n-1}\cdot \nabla
_{v}f^{n}& =0\;\;\;\;,\;\;\;\;x\in \Omega \subset \mathbb{R}%
^{3}\;\;\;,\;v\in \mathbb{R}^{3}\;\;,\;\;t>0  \label{itfn1} \\
\Delta \phi ^{n-1}& =\rho ^{n-1}\left( x\right) \equiv \int_{\mathbb{R}%
^{3}}f^{n-1}dv\;\;,\;\;x\in \Omega \;\;,\;\;t>0  \label{itfn2} \\
\phi ^{n-1}& =0\;\;,\;\;x\in \partial \Omega \;\;,\;\;t>0  \label{itfn3} \\
f^{n}\left( 0,x,v\right) & =f_{0}\left( x,v\right) \;\;\;x\in \Omega
\;\;,\;\;v\in \mathbb{R}^{3}  \label{itfn4} \\
f^{n}\left( t,x,v\right) & =f^{n}\left( t,x,v^{\ast }\right) \;\;x\in
\partial \Omega \;\;,\;\;v\in \mathbb{R}^{3}\;\;,\;\;t>0  \label{itfn5}
\end{align}%
for $n=1,2,....\;$We assume that $f_{0}$ satisfies the nonnegativity
condition as well as (\ref{S2E0b})-(\ref{S2flatness}).

\bigskip We will use the notation
\begin{equation}
E^{n}=\nabla \phi ^{n}.  \label{itfield}
\end{equation}

The goal is to show that the sequence $f^{n}$ converges as $n\rightarrow
\infty $ for all $0\leq t<\infty .$ To this end we need to show as a first
step that this sequence is globally defined in time for each $n\geq 0.$

\subsection{The iterative sequence $\left\{ f^{n}\right\} $ is globally
defined in time}

\bigskip

Given a function $g:\Omega \rightarrow \mathbb{R},$ we will denote as $\left[
\cdot \right] _{0,\lambda ;x}$ the seminorm
\begin{equation*}
\left[ g\right] _{0,\lambda ;x}\equiv \sup_{x,y\in \Omega }\frac{\left\vert
g\left( x\right) -g\left( y\right) \right\vert }{\left\vert x-y\right\vert
^{\lambda }}.
\end{equation*}

We define

\begin{equation}
Q\left( t\right) \equiv \sup \left\{ \left\vert v\right\vert ~|~\left(
x,v\right) \in \text{supp~}f\left( s\right) ,\text{ \ }0\leq s\leq t\right\}
.  \label{Q}
\end{equation}

\bigskip

\begin{prop}
Let $\mu ,\lambda \in \left( 0,1\right) $, satisfying $\mu >\lambda .$ Let $%
f_{0}\in C_{0}^{1,\mu }\left( \bar{\Omega}\times \mathbb{R}^{3}\right) ,$ $%
f_{0}\geq 0$ satisfy (\ref{S2flatness}). Then, the sequence of functions $%
f^{n}$ is globally defined for each $x\in \Omega ,\;v\in \mathbb{R}^{3}$ and
$0\leq t<\infty .$ Moreover we have $f^{n}\in C_{t;\left( x,v\right)
}^{1;1,\lambda }\left( \left[ 0,T\right] \times \Omega \times \mathbb{R}%
^{3}\right) $ for any $T>0$ and $\left\Vert f^{n}\right\Vert _{\infty
}=\left\Vert f_{0}\right\Vert _{\infty },\;\int \rho _{n}\left( x,t\right)
dx=\int f_{0}\left( t,x,v\right) dxdv.$
\end{prop}

\begin{proof}
This proposition can be proved using Theorem \ref{Thlinear} and by induction
on $n$. We omit the details.
\end{proof}

\bigskip

\subsection{The sequence $\left\{ f^{n}\right\} $ converges to a solution of
the VP system if the sequence $\left\{ Q^{n}\right\} $ is bounded.}

\bigskip

We define the following measure for the maximal velocities reached for the
distribution $f^{n}$

\begin{equation}
Q^{n}\left( t\right) \equiv\sup\left\{ \left| v\right| ~|~\left( x,v\right)
\in\text{supp~}f^{n}\left( s\right) ,\text{ \ }0\leq s\leq t\right\} .
\label{Qndef}
\end{equation}

\begin{prop}
\label{Propconv}Under the assumptions of Theorem \ref{globalexistence},
suppose that $Q^{n}\left( t\right) \leq K$ for $n\geq n_{0},\;0\leq t\leq T.$
Then, $f^{n}\rightarrow f$ in $C_{t;\left( x,v\right) }^{\nu ;1,\lambda
}\left( \left[ 0,T\right] \times \Omega \times \mathbb{R}^{3}\right) $ as $%
n\rightarrow \infty $ with $0<\lambda <\mu ,\;0<\nu <1$ and where $f\in
C_{t;\left( x,v\right) }^{1;1,\lambda }\left( \left[ 0,T\right] \times
\Omega \times \mathbb{R}^{3}\right) $ is a solution of (\ref{S1E1})-(\ref%
{S1E5}).
\end{prop}

\begin{proof}
The proof of the Proposition \ref{Propconv} is similar to that of
Proposition 3 in \cite{HJHJJLV2} and we omit it.
\end{proof}

\bigskip

\subsection{Prolongability of uniform estimates for the functions $f^{n}$}

\bigskip

\begin{prop}
\label{Qconv}Let $Q^{n},\;Q$ be as in (\ref{Qndef}), (\ref{Q}) respectively.
Suppose that $\max \left\{ \sup_{n\geq n_{0}}Q^{n}\left( t\right) ,Q\left(
t\right) \right\} \leq K$ for $0\leq t\leq T.$ We also assume that $%
f^{n}\rightarrow f$ in $C_{t;\left( x,v\right) }^{\nu ;1,\lambda }\left( %
\left[ 0,T\right] \times \Omega \times \mathbb{R}^{3}\right) $ for any $%
0<\lambda <\mu ,\;0<\nu <1.$ Then $\lim_{n\rightarrow \infty }Q^{n}\left(
t\right) =Q\left( t\right) \;$uniformly on $\left[ 0,T\right] .$
\end{prop}

\begin{proof}
The proof relies on Velocity lemma, Lemma \ref{velocity}. The
characteristics starting in $\alpha \left( 0\right) \geq C\delta _{0}$
remain during their evolution in the set $\left\{ \alpha \left( t\right)
\geq C\delta _{0}\right\} $ due to Lemma \ref{velocity}. Therefore, these
characteristics remain separated from the singular set from which we can
deduce the convergence of $Q^{n}\left( t\right) $ to $Q\left( t\right) $ as $%
n\rightarrow \infty .$ For the details, refer to \cite{HJHJJLV2}.
\end{proof}

\bigskip

The following Proposition concerns the prolongability of the uniform
estimates on $Q^{n}\left( t\right) $ and we skip the proof.

\begin{prop}
\label{Prolog}Suppose that for some $T\geq 0$ there exist $K>0$ and $%
n_{0}\geq 0$ such that for any $n\geq n_{0}$ and $0\leq t\leq T$ we have $%
Q^{n}\left( t\right) \leq K.$ Then, there exists $\varepsilon
_{0}=\varepsilon _{0}\left( K,\left\Vert f_{0}\right\Vert _{\infty }\right)
>0$ such that for $0\leq t\leq T+\varepsilon _{0}$ and $n\geq n_{0}$ the
following estimate holds
\begin{equation*}
Q^{n}\left( t\right) \leq 2K.
\end{equation*}
\end{prop}

We give in the following some of basic energy estimates for the
Vlasov-Poisson system (cf. \cite{RG}). Proofs are standard in kinetic theory
and we omit them.

\begin{prop}
Suppose that $f$ is a solution of (\ref{S1E1})-(\ref{S1E5}) defined in $%
0\leq t\leq T$ with $f\left( 0,x,v\right) =f_{0}\left( x,v\right) $, where $%
f_{0}\geq 0$. There exists $C$ depending only on $T$ and on the regularity
norms assumed for $f_{0}$ in Theorem \ref{globalexistence} such that
\begin{align}
\sup_{0\leq t\leq T}\int_{\Omega }v^{2}f\left( x,t\right) dvdx& \leq C
\label{S2E1a} \\
\sup_{0\leq t\leq T}\left[ \left\Vert \rho \left( t,\cdot \right)
\right\Vert _{L^{\frac{5}{3}}\left( \Omega \right) }\right] & \leq C
\label{S2E1b}
\end{align}%
\begin{equation}
\left\Vert f\left( t\right) \right\Vert _{L^{p}\left( \Omega \times \mathbb{R%
}^{3}\right) }=\left\Vert f_{0}\right\Vert _{L^{p}\left( \Omega \times
\mathbb{R}^{3}\right) }\;,\text{ \ for all }1\leq p\leq \infty .
\label{S2E2b}
\end{equation}%
\begin{equation}
\frac{d}{dt}\left( \int_{\Omega \times \mathbb{R}^{3}}v^{2}fdxdv+\int_{%
\Omega }\left( E\right) ^{2}dx\right) =0  \label{Eniter}
\end{equation}
\end{prop}

\bigskip

\section{\protect\bigskip Global bound for $Q\left( t\right) .$}

In this section we show that the function $Q\left( t\right) $ can be bounded
in any time interval $0\leq t\leq T$ and therefore that the corresponding
solutions of (\ref{S1E1})-(\ref{S1E5}) can be extended to arbitrarily long
intervals. The global-in-time bound on $Q\left( t\right) $ was first proved
by Pfaffelmoser (cf. \cite{KP}) in the case of the whole space and the
method of Pfaffelmoser has been adapted to the case of bounded domains with
purely reflected boundary conditions at $\partial \Omega $ (cf. \cite%
{HJHJJLV2}). The main content of the result is a uniform estimate for $%
Q\left( t\right) $ as long as $f$ is defined.

From the definition (\ref{Q}) of $Q\left( t\right) ,$ we obtain the
following estimate
\begin{equation}
\left\Vert \rho \right\Vert _{\infty }\leq \left\Vert f\right\Vert _{\infty
}Q\left( t\right) ^{3}.  \label{rhoestimate}
\end{equation}%
where $\rho $ is in (\ref{S1E2}).

The main result of this section is given in the following. Since the theorem
and its proof do not depend on the boundary conditions for the electric
potential $\phi $ (whether Dirichlet or Neumann), they can be proved as in
Theorem 3 in \cite{HJHJJLV2}. We state the theorem without its proof and
without auxiliary lemmas.

\begin{thm}
\label{boundforQ}Let $f_{0}\in C^{1,\mu }\left( \Omega \times \mathbb{R}%
^{3}\right) $ with $0<\mu <1.$ Suppose that $f\in C_{t,\left( x,v\right)
}^{1;1,\lambda }\left( [0,T]\times \Omega \times \mathbb{R}^{3}\right) $ is
a solution of (\ref{S1E1})-(\ref{S1E5}) with $\lambda \in \left( 0,1\right)
,\;0<T<\infty .$ There exists $\sigma \left( T\right) <\infty $ depending
only on $T,\;Q\left( 0\right) ,$ and $\left\Vert f_{0}\right\Vert _{C^{1,\mu
}\left( \Omega \times \mathbb{R}^{3}\right) }$ such that
\begin{equation}
Q\left( t\right) \leq \sigma \left( T\right) \;\;,\;\;0\leq t\leq T\;.\;
\label{Qbound}
\end{equation}
\end{thm}

\bigskip

\begin{proof}[\textbf{Proof of Theorem \protect\ref{globalexistence}}]
The proof is similar to that in \cite{HJHJJLV2} and we omit it.
\end{proof}

\bigskip

\textbf{Acknowledgement} H.J. Hwang is supported by Basic Science Research
Program through the National Research Foundation of Korea funded by the
Ministry of Education, Science and Technology (2010-0008127) and also by
Priority Research Centers Program (2009-0094068). J.J.L. Vel\'{a}zquez is
supported by the grant DGES MTM2007-61755.

\end{document}